\documentclass[11pt]{article}

\usepackage{graphicx}
\usepackage{epsfig}
\usepackage{amsmath, amsthm, amssymb, mathtools}
\usepackage{enumitem}
\usepackage[mathscr]{euscript}
\usepackage[english]{babel}
\usepackage{aliascnt}
\usepackage[all,cmtip]{xy}
\usepackage{url}
\usepackage{tikz-cd}
\usepackage{hyperref}
\usepackage[capitalize]{cleveref}
\usepackage{stmaryrd}

\def\ds{\displaystyle}

\newcommand\A[1]{\mathrm{A}_{#1}}

\def\Spec{\mathrm{Spec}}
\def\Hom{\mathrm{Hom}}

\def\det{\mathrm{det}}
\def\norm{\mathrm{Nm}}

\def\trace{\mathrm{Tr}}

\def\NN{\mathbb{N}}

\newcommand{\B}[1]{\operatorname{\mathbf{B}{#1}}}

\newcommand{\extpower}{{\textstyle\bigwedge}}

\def\eps{\varepsilon}

\def\simto{\ds\mathop{\longrightarrow}^\sim\,}

\renewcommand\A{\mathcal{A}}
\renewcommand\B{\mathcal{B}}
\newcommand\C{\mathcal{C}}
\newcommand\D{\mathcal{D}}
\renewcommand\L{\mathcal{L}}
\newcommand\G{\mathcal{G}}
\newcommand\M{\mathcal{M}}

\newcommand{\disc}{\delta}

\newcommand{\stdi}[1]{\tau}

\newcommand{\quadalg}[3][]{(\!({#2} : {#3}]\hspace{-0.07em}]_{{#1}}}
\newcommand{\bigquadalg}[3][]{\left(\!\!\left({#2} : {#3}\right]\!\right]_{{#1}}}

\newtheorem{theorem}{Theorem}[section] % numbered by section
\crefname{theorem}{Theorem}{Theorems}

\newtheorem{lemma}[theorem]{Lemma}
\crefname{lemma}{Lemma}{Lemmas}

\newtheorem{proposition}[theorem]{Proposition}
\crefname{proposition}{Proposition}{Propositions}

\crefname{corollary}{Corollary}{Corollaries}

\newtheorem{conjecture}[theorem]{Conjecture}
\crefname{conjecture}{Conjecture}{Conjectures}

\theoremstyle{definition} % styled differently: not italicized

\newtheorem{definition}[theorem]{Definition}
\crefname{definition}{Definition}{Definitions}

\newtheorem{remark}[theorem]{Remark}
\crefname{remark}{Remark}{Remarks}

\newtheorem{example}[theorem]{Example}
\crefname{example}{Example}{Examples}

\author{Owen Biesel\footnote{Carleton College, Northfield, MN. Email: obiesel@carleton.edu}}
\begin{document}
\title{A Norm Functor for Quadratic Algebras}
\maketitle
\begin{abstract}
Given commutative, unital rings $\A$ and $\B$ with a ring homomorphism $\A\to\B$ making $\B$ free of finite rank as an $\A$-module, we can ask for a ``trace'' or ``norm'' homomorphism taking algebraic data over $\B$ to algebraic data over $\A$. In this paper we we construct a norm functor for the data of a quadratic algebra: given a locally-free rank-$2$ $\B$-algebra $\D$, we produce a locally-free rank-$2$ $\A$-algebra $\norm_{\B/\A}(\D)$ in a way that is compatible with other norm functors and which extends a known construction for étale quadratic algebras. We also conjecture a relationship between discriminant algebras and this new norm functor.
\end{abstract}
\tableofcontents

\section{Introduction}

Given a homomorphism of commutative, unital rings $\A\to\B$ that makes $\B$ locally free of finite rank $n$ as an $\A$-module, there is a zoo of ``trace'' or ``norm'' operations taking algebraic data over $\B$ to algebraic data of the same type over $\A$. For example:

\begin{itemize}
 \item Given an element $b\in \B$, we can take its \emph{trace} to get an element of $\A$: working locally so that $\B$ has an $\A$-basis, represent multiplication by $b$ as the action of a square matrix over $\A$, and then take the trace of that matrix. This gives us an $\A$-linear function $\trace_{\B/\A}:\B\to\A$.
 \item Alternatively, we can take the \emph{norm} of $b$ by taking the determinant of the matrix by which $b$ acts. This gives a multiplicative function $\norm_{\B/\A}: \B\to\A$. 
 \item Given a line bundle (a locally free rank-$1$ module) $\L$ over $\B$, we can take its ``norm'' to get $\norm_{\B/\A}(\L) = \Hom_{\A}(\extpower^n \B, \extpower^n \L)$, a line bundle over $\A$. We can also apply this norm operations to homomorphisms of line bundles $\L\to \L'$, and furthermore, if an endomorphism $f:\L\to\L$ is given by multiplication by $b\in\B$, then the endomorphism $\norm_{\B/\A}(f)$ of $\norm_{\B/\A}(\L)$ is multiplication by the ordinary norm of $b$ in $\A$.
 \item More generally, in \cite{Ferrand} Ferrand shows how to take an arbitrary $\B$-module $\M$ and construct its ``norm'' $\norm_{\B/\A}(M)$ as an $\A$-module. In the special case that $\M$ is a line bundle this functor agrees with the above definition of $\norm_{\B/\A}(M)$, but in general if $\M$ is locally free it does not follow that $\norm_{\B/\A}(\M)$ is locally free.
\end{itemize}

In each case, there is a group or monoid operation that the trace or norm preserves. The ordinary trace preserves addition: $\trace_{\B/\A}(b+b') = \trace_{\B/\A}(b) + \trace_{\B/\A}(b')$, while the ordinary norm preserves multiplication $\norm_{\B/\A}(bb') = \norm_{\B/\A}(b)\norm_{\B/\A}(b')$. On the other hand, the norm for line bundles and other modules is a functor that preserves tensor products up to isomorphism: $\norm_{\B/\A}(\M\otimes_\B \M') \cong \norm_{\B/\A}(\M)\otimes_\A \norm_{\B/\A}(\M')$.

In case the map of schemes $\pi: \Spec(\B)\to\Spec(\A)$ is \'etale of degree $n$ (meaning that after a faithfully-flat finite-presentation base change it becomes a trivial degree-$n$ cover of the form $\coprod_{i=1}^n \Spec(\A)\to\Spec(\A)$), then for any sheaf $\G$ of abelian groups on the big \'etale site over $\Spec(\A)$ we have a ``trace'' homomorphism $\pi_\ast\pi^\ast \G \to \G$, for which taking sections over $\Spec(\A)$ gives us a homomorphism $\G(\B)\to\G(\A)$. If $\G$ is the additive group $\mathbb{G}_{a, \A}$, represented by the scheme $\Spec(\A[x])$, then this homomorphism is just the ordinary trace map $\B\to\A$. If $\G$ is the multiplicative group $\mathbb{G}_{m,\A}$ represented by $\Spec(\A[x,x^{-1}])$, then we get the norm map on units $\B^\times\to\A^\times$. It is not hard to extend the argument in \cite[Lemma V.1.12]{Milne} to sheaves of commutative monoids, giving us the full multiplicative norm map $\B\to\A$.

Meanwhile, for the norm of line bundles, one way to think about what is going on in the case of $\A\to\B$ \'etale is that line bundles are torsors for $\mathbb{G}_m$, and so we are applying $H^1(\Spec(\A), \cdot)$ to the ``trace'' homomorphism $\pi_\ast\pi^\ast \mathbb{G}_{m,\A}\to\mathbb{G}_{m,\A}$, giving a function
\begin{align*}
H^1(\Spec(\A), \pi_\ast\pi^\ast \mathbb{G}_{m,\A}) &\to H^1(\Spec(\A), \mathbb{G}_{m,\A})\\
\text{i.e. }H^1(\Spec(\B), \mathbb{G}_{m,\B}) &\to H^1(\Spec(\A), \mathbb{G}_{m,\A})
\end{align*}
sending the isomorphism class of a line bundle over $\B$ to the isomorphism class of its norm as a line bundle over $\A$. This, however, forgets the functorial nature of the norm operation on line bundles. Meanwhile, Ferrand's extension of the norm functor to all $\B$-modules is defined in terms of a universal property ($\norm_{\B/\A}(\M)$ is the universal $\A$-module equipped with a homogeneous degree-$n$ polynomial law $\M\to\norm_{\B/\A}(\M)$) and does not readily admit a cohomological interpretation even in the \'etale case.

However, in \cite{Waterhouse} Waterhouse uses the cohomological method to define a ``trace'' operation for \'etale quadratic algebras, since they are parameterized by $S_2$-torsors. Explicitly, if $\A\to\B$ is a rank-$n$ \'etale algebra in which $2$ is a unit, we can describe this ``trace'' operation as sending each \'etale quadratic $\B$-algebra of the form $\B[\sqrt{d}]$, with $d\in\B^\times$, to the \'etale quadratic $\A$-algebra $\A\left[\sqrt{\norm_{\B/\A}(d)}\right]$. It is the goal of the present paper to extend this construction to a norm functor from the category of quadratic $\B$-algebras to the category of quadratic $\A$-algebras, regardless of whether any of the algebras are \'etale or whether $2$ is a unit. This norm operation is defined for quadratic algebras with a chosen generator in \cref{def:norm-based}, made functorial in \cref{prop:different-generator}, and extended to arbitrary quadratic algebras in \cref{def:norm}.

The main theorems about this norm functor for quadratic algebras are:
\begin{enumerate}
 \item It commutes with base change (\cref{thm:base-change}).
 \item It is ``transitive'' in the sense that for a tower of algebras $\A\to\B\to\C$, taking the norm of a quadratic $\C$-algebra to get a quadratic $\B$-algebra, then taking the norm again to get a quadratic $\A$-algebra, gives the same result as regarding $\C$ as an $\A$-algebra and applying the norm operation once (\cref{thm:transitive}).
 \item It extends the notion of trace of an $S_2$-torsor in the case that the algebras are \'etale (\cref{thm:etale-addition}).
 \item For the unique extension of $S_2$-torsor addition to a monoid operation $\ast$ on all quadratic algebras characterized by Voight in \cite{Voight}, the norm functor is a monoid homomorphism with respect to $\ast$ (\cref{thm:homomorphism}).
 \item Taking the norm of a quadratic algebra commutes with taking its determinant line bundle and discriminant bilinear form (\cref{thm:norm-det} and \cref{thm:norm-disc}).
\end{enumerate}

However, it is still an open question whether this notion of the norm of a quadratic algebra is compatible with the notion of discriminant algebra from \cite{BieselGioia} in the sense of \cite{Waterhouse}; see \cref{conj:norm-discalg}.

Acknowledgements: The author would like to thank John Voight and Asher Auel for careful reading and helpful comments on an early draft, and Marius Stekelenburg for a much simpler proof of \cref{lem:transitive-sn}.

\section{Finite-rank and Quadratic Algebras}

We begin with background and notation for concepts related to rank-$n$ algebras in general and quadratic algebras in particular. In the following, all rings and algebras are commutative and unital, and are denoted by calligraphic capital letters: $\A, \B,\C,\D,\dots$.

\begin{definition}
 Let $\A$ and $\B$ be rings and $\A\to\B$ a ring homomorphism, so that $\B$ is an $\A$-algebra. We say that $\B$ is a \emph{rank-$n$ $\A$-algebra} if $\B$ is projective of constant rank $n$ as an $\A$-module; equivalently, if there exist $a_1,\dots,a_k\in\A$ together generating the unit ideal such that each localization $\B_{a_i}$ is isomorphic to $\A_{a_i}^{\oplus n}$ as $\A_{a_i}$-modules.
\end{definition}
\begin{definition}\label{def:quadalg}
 A \emph{quadratic $\A$-algebra} $\D$ is an $\A$-algebra of rank $2$. If $\D$ is equipped with a choice of algebra generator $x$, so that $\D\cong \A[x]/(x^2 - tx + n)$ for some elements $t,n\in\A$ (the \emph{trace} and \emph{norm} of $x\in\D$), then we call $\D$ a \emph{based} quadratic algebra. (Not every quadratic algebra admits a singleton generating set, but every quadratic algebra does so locally.) As a based quadratic algebra is determined up to unique isomorphism by the ordered pair $(t,n)$, we follow \cite{Loos} in using the notation
 \[\quadalg[\A]{t}{n} \coloneqq \A[x]/(x^2 - tx + n)\]
 for based quadratic algebras (the subscript ring may be omitted if it is clear from context). For example, $\quadalg[\A]{1}{0} = \A[x]/(x^2 - x) \cong \A\times\A$ with generating element $(1,0)$, and $\quadalg[\A]{0}{0}\cong \A[\eps]/(\eps^2)$ with generator $\eps$.
\end{definition}

In \cref{def:quadalg}, we referred to the trace and norm of an element. We will not continue to need the notion of trace of a general element of a rank-$n$ algebra, but we will make frequent use of elements' norms. Because there are so many notions of ``norm'' we will be using in this paper, we will reserve the notation $\norm_{\B/\A}$ for norm functors, applied to mathematical objects such as line bundles and other modules, and soon quadratic algebras. To refer to the ordinary norm function for a rank-$n$ algebra $\A\to\B$, we use the notation $s_n$ as defined below:

\begin{definition}\label{def:s_n}
 Let $\A\to\B$ be a rank-$n$ algebra. Then there is a canonical function $s_n\colon \B\to\A$, called the \emph{norm}, defined as follows:
 \begin{enumerate}
  \item If $\B$ has a basis as an $\A$-module, and if $b$ is an element of $\B$, then its norm $s_n(b)$ is the determinant of the map $\B\to\B$ given by multiplication by $b$ (which is independent of choice of basis).
  \item If $\B$ is a general rank-$n$ $\A$-algebra, and if $b\in\B$, then we can define the norm of $b$ in any localization $\B_a$ that is free as an $\A_a$-module. These values of $s_n(b)$ in each such $\A_a$ glue to give a single well-defined value for $s_n(b)\in \A$.
 \end{enumerate}
\end{definition}
\begin{remark} The norm function $s_n\colon \B\to\A$ is not generally a ring homomorphism, but rather has the property that if $a\in\A$ is any element of the base ring, then $s_n(ab) = a^n s_n(b)$. Furthermore, the calculation of $s_n(b)$ commutes with base change, making $s_n$ an example of a \emph{polynomial law} in the sense of \cite{Roby}. For us, the upshot of $s_n$ being a polynomial law is that it automatically comes with polarized versions $s_{k_1,k_2,\ldots, k_m}(b_1,b_2,\ldots b_m)$ for each list of natural numbers $k_1+\dots+k_m = n$, defined by
 \[s_{k_1,k_2,\ldots, k_m}(b_1,b_2,\ldots b_m) \coloneqq \begin{array}{cc}
 \text{the coefficient of }\lambda_1^{k_1}\lambda_2^{k_2}\cdots\lambda_{m}^{k_m}\text{ in }\\s_n(\lambda_1 b_1+\lambda_2 b_2 + \ldots + \lambda_m b_m)\end{array},\]
 the latter of which is calculated via the norm map of the base changed rank-$n$ algebra $\A[\lambda_1,\dots,\lambda_m]\to\B[\lambda_1,\dots,\lambda_m]$.
Phrased another way, the $s_{k_1,\dots,k_m}$ are defined so that we have the following identity:
\[s_n(\lambda_1b_1+\ldots+\lambda_mb_m) = \sum_{\substack{k_1,\dots,k_m\in\NN\\
k_1+\ldots+k_m = n}}\lambda_1^{k_1}\dots \lambda_m^{k_m} s_{k_1,\dots,k_m}(b_1,\dots,b_m).\]
 For example, the coefficients of the characteristic polynomial of an element $b\in\B$ can be expressed in terms of the polarized forms of the norm function:
 \begin{align*}
  s_n(\lambda - b) &= \left.\vphantom{\sum}s_n(\lambda 1 + \mu b)\right|_{\mu=-1}\\
  &= \left.\sum_{k=0}^n \lambda^k\mu^{n-k} s_{k,n-k}(1,b)\right|_{\mu=-1}\\
  &= \sum_{k=0}^n \lambda^k (-1)^{n-k} s_{k, n-k}(1,b).
 \end{align*}
  In particular, the \emph{trace} of $b\in \B$ can be calculated as the trace of the matrix representing multiplication by $b$ in any localization making $\B$ a free $\A$-module, or it can be expressed as the polarized form $s_{1,n-1}(b, 1)$.
\end{remark}

We now prove some of the basic results about the norm functions $s_n$ that we will need later:

\begin{lemma}\label{lem:multiplicative-sn}
 The norm functions are multiplicative: If $\A\to\B$ is a rank-$n$ algebra and $b,b'\in\B$, then $s_n(bb') = s_n(b)s_n(b')$ in $\A$.
\end{lemma}

\begin{proof}
 We can check locally, so assume that $\B$ has an $\A$-basis. Then the $n\times n$ matrix by which multiplication by $bb'$ acts, with respect to this basis, is the product of the matrices by which $b$ and $b'$ act. Taking determinants of both sides, and since the determinant is multiplicative, we obtain $s_n(bb') = s_n(b)s_n(b')$.
\end{proof}

\begin{lemma}\label{lem:transitive-sn}
 The norm functions are transitive: Let $\A\to\B$ be a rank-$n$ algebra and $\B\to\C$ be a rank-$m$ algebra, so that $\C$ is also a rank-$mn$ algebra. For all $c\in\C$, we have $s_n(s_m(c)) = s_{mn}(c)$ as elements of $\A$.
\end{lemma}

\begin{proof}
 We can work locally on $\A$, so assume that $\B$ is free as an $\A$-module. By \cite[Prop.\ 6.1.12 on p.\ 113]{EGAII}, since $\A\to\B$ is finite we can also assume that $\C$ is free as a $\B$-module by localizing $\A$. So assume that $\B$ is free of rank $n$ as an $\A$-module and that $\C$ is free of rank $m$ as a $\B$-module.
  Choose an $\A$-basis $\theta_1,\dots,\theta_n$ for $\B$ and a $\B$-basis $\phi_1,\dots,\phi_m$ for $\C$. 
  
  First consider the $m\times m$ matrix $(M_{ij})_{i,j}$ with elements in $\B$ representing multiplication by $c$ with respect to basis $\phi_1,\dots,\phi_m$, i.e.\ 
  \[c\phi_i = \sum_{j=1}^m M_{ij}\phi_j.\]
  Then for each $M_{ij}\in\B$, consider the $n\times n$ matrix $(P_{ijk\ell})_{k\ell}$ representing multiplication by $M_{ij}$ with respect to basis $\theta_1,\dots,\theta_n$, i.e.\ 
  \[M_{ij}\theta_k = \sum_{\ell=1}^n P_{ijk\ell} \theta_\ell.\]
  Then the $mn\times mn$ matrix $(P_{ijk\ell})_{(i,k), (j,\ell)}$, with rows and columns indexed by $\{1,\dots,m\}\times\{1,\dots,n\}$, represents multiplication by $c$ with respect to the $\A$-basis $\phi_1\theta_1, \phi_2\theta_1, \dots, \phi_m\theta_n$ for $\C$:
  \begin{align*}
   c \phi_i\theta_k &= \sum_{j=1}^m M_{ij}\phi_j\theta_k\\
   &= \sum_{j=1}^m \sum_{\ell=1}^n P_{ijk\ell}\phi_j\theta_\ell\\
   &= \sum_{(j,\ell)\in\{1,\dots,m\}\times\{1,\dots,n\}} P_{ijk\ell} \phi_j\theta_\ell.
  \end{align*}
We can regard $(P_{ijk\ell})_{(i,k), (j,\ell)}$ as an $m\times m$ block matrix, where the $n\times n$ blocks that represent multiplication by $M_{ij}$ all commute with each other (since the elements $M_{ij}\in\B$ all commute, as $\B$ is a commutative algebra). Then by \cite{Kovacs}, the determinant of $(P_{ijk\ell})_{(i,k), (j,\ell)}$ is equal to the determinant of the $n\times n$ matrix given by formally applying the $m\times m$ determinant formula to the blocks of $(P_{ijk\ell})_{(i,k), (j,\ell)}$. In other words,
\begin{align*}
 s_{mn}(c) &= s_n(\det((M_{ij})_{i,j}))= s_n(s_m(c)).\qedhere
\end{align*}
\end{proof}

Here is a collection of results about the polarized forms of $s_n$:

\begin{lemma}\label{lem:polarized-identity}
 Let $\A\to\B$ be a rank-$n$ algebra, let $n = k_1+\ldots+k_m$ be a partition of $n$, and let $b_1,\dots,b_m\in\B$. Then the following identities hold:
 \begin{enumerate}
  \item (Reordering) If $\sigma:\{1,\dots,m\}\to\{1,\dots,m\}$ is any permutation, then
  \[s_{k_1,\dots,k_m}(b_1,\dots,b_m) = s_{k_{\sigma(1)},\dots,k_{\sigma(m)}}(b_{\sigma(1)},\dots,b_{\sigma(m)}).\]
  \item (Combination) If $b_1 = b_2$, then
  \[s_{k_1,\dots,k_m}(b_1,\dots,b_m) = \binom{k_1+k_2}{k_1} s_{k_1+k_2,k_3,\dots,k_m}(b_1,b_3,\dots,b_m).\]
  \item (Homogeneity) If $a\in\A$, then
  \[s_{k_1,\dots,k_m}(ab_1,\dots,b_m) = a^{k_1} s_{k_1,\dots,k_m}(b_1,\dots,b_m).\]
  \item (Degeneracy) If $k_1=0$, then
  \[s_{k_1,k_2,\dots,k_m}(b_1,b_2,\dots,b_m) = s_{k_2,\dots,k_m}(b_2,\dots,b_m).\]
  \item (Multiplicativity) If $b'\in\B$, then
  \[ s_{k_1,\dots,k_m}(b'b_1,\dots,b'b_m) = s_n(b') s_{k_1,\dots,k_m}(b_1,\dots,b_m).\]
 \end{enumerate}
\end{lemma}

\begin{proof}
 \begin{enumerate}
  \item (Reordering) This follows from comparing the coefficients of $\lambda_1^{k_1}\dots \lambda_m^{k_m} = \lambda_{\sigma(1)}^{k_{\sigma(1)}}\dots \lambda_{\sigma(m)}^{k_{\sigma(m)}}$ in $s_n(\lambda_1b_1+\dots+\lambda_m b_m) = s_n(\lambda_{\sigma(1)}b_{\sigma(1)} + \dots + \lambda_{\sigma(m)}b_{\sigma(m)})$.
  \item (Combination) The term $s_{k_1,\dots,k_m}(b_1,\dots,b_m)$ is by definition the coefficient of $\lambda_1^{k_1}\dots \lambda_m^{k_m}$ in
  \begin{align*}
   s_n(\lambda_1 b_1 + \ldots + \lambda_m b_m) &= s_n\bigl((\lambda_1+\lambda_2)b_1+\ldots + \lambda_m b_m\bigr)\\
   &= \sum_{\substack{\ell, \ell_3, \dots, \ell_m\in \NN:\\ \ell+\ell_3+\ldots+\ell_m = n}} (\lambda_1+\lambda_2)^\ell \lambda_3^{\ell_3}\dots\lambda_m^{\ell_m} s_{\ell, \ell_3, \ldots, \ell_m}(b_1,b_3,\dots,b_m),
  \end{align*}
  in which the coefficient of $\lambda_1^{k_1}\lambda_2^{k_2}\dots \lambda_m^{k_m}$ is $\binom{k_1+k_2}{k_1} s_{k_1+k_2,k_3,\dots,k_m}(b_1,b_3,\ldots,b_m)$, as desired.
  \item (Homogeneity) The quantity $s_{k_1,\dots,k_m}(ab_1,\dots,b_m)$ is the coefficient of $\lambda_1^{k_1}\dots\lambda_m^{k_m}$ in
  \[s_n(a\lambda_1b_1+\lambda_2b_2+\dots+\lambda_m b_m) = \sum_{\substack{\ell_1,\dots,\ell_m\in\NN\\\ell_1+\ldots+\ell_m = n}} (a\lambda_1)^{\ell_1}\lambda_2^{\ell_2}\dots \lambda_m^{\ell_m} s_{\ell_1,\dots,\ell_m}(b_1,\dots,b_m),\]
  so the coefficient of $\lambda_1^{k_1}\lambda_2^{k_2}\dots \lambda_m^{k_m}$ is $a^{k_1}s_{k_1,\dots,k_m}(b_1,\dots,b_m)$, as desired.
  \item (Degeneracy) If $k_1 = 0$, then we are looking for the coefficient of $\lambda_2^{k_2}\dots\lambda_m^{k_m}$ in $s_n(\lambda_1b_1 + \dots + \lambda_m b_m)$, which since it does not involve $\lambda_1$ is unchanged if we base change to give $\lambda_1$ a concrete value. In particular, we can base change along $\A[\lambda_1,\dots,\lambda_m]\to\A[\lambda_2,\dots,\lambda_m] : \lambda_1\mapsto 0$, giving us the coefficient of $\lambda_2^{k_2}\dots\lambda_m^{k_m}$ in $s_n(\lambda_2b_2 + \dots + \lambda_m b_m)$, namely, $s_{k_2,\dots,k_m}(b_2,\dots,b_m)$.
  \item We have that $s_{k_1,\dots,k_m}(b'b_1,\dots,b'b_m)$ is the coefficient of $\lambda_1^{k_1}\dots \lambda_m^{k_m}$ in 
  \[s_n(\lambda_1b'b_1 + \ldots +\lambda_m b' b_m) = s_n(b') s_n(\lambda_1 b_1 + \ldots + \lambda_m b_m)\]
  by \cref{lem:multiplicative-sn}.
  The coefficient of $\lambda_1^{k_1}\dots \lambda_m^{k_m}$ is therefore $s_n(b') s_{k_1,\dots,k_m}(b_1,\dots,b_m)$.
 \end{enumerate}
\end{proof}

Now we consider the various senses in which rank-$n$ algebras over $\A$ form not just a set of isomorphism classes, but a category:

\begin{definition}
 There are various levels of strictness for homomorphisms of rank-$n$ $\A$-algebras. The weakest is that of ordinary $\A$-algebra homomorphism. A stronger notion that uses the rank-$n$ structure is that of a \emph{norm-preserving} homomorphism, which is an $\A$-algebra homomorphism $\B\to\B'$ that, together with the norm functions $s_n$, makes the resulting triangle of functions
 \[\begin{array}{ccccc}
  \B & & \to & & \B'\\
  & \searrow & & \swarrow\\
  & & A & & 
 \end{array}\]
 commute. (For example, if $\A$ is any nonzero ring, the $\A$-algebra homomorphism $\A\times\A \to\A\times\A$ sending $(a,a')$ to $(a,a)$ is not norm-preserving, since $(1,0)$ has norm $0$ and $(1,1)$ has norm $1$.)
 
 There is a stronger notion of \emph{universally norm-preserving} homomorphism, which is a norm-preserving $\A$-algebra homomorphism that remains so after arbitrary base change. Isomorphisms of rank-$n$ algebras are always universally norm-preserving, though the converse is not true.
\end{definition}

However, we can say more in case our algebras are quadratic:

\begin{lemma}
 Let $\A$ be a ring and $\quadalg[\A]{t}{n}$ and $\quadalg[\A]{t'}{n'}$ be two based quadratic algebras. For a homomorphism $f:\quadalg[\A]{t'}{n'}\to\quadalg[\A]{t}{n}$ sending $x\mapsto ux+c$, the following are equivalent:
 \begin{enumerate}
 \item $f$ is universally norm-preserving,
 \item $f$ is norm-preserving,
 \item $t' = ut+2c$ and $n' = u^2n + uct + c^2$.
\end{enumerate}
\end{lemma}
In particular, since quadratic algebras can all locally be given a based structure, all norm-preserving homomorphisms of quadratic algebras are universally norm-preserving.

\begin{proof}
 $(1)\implies (2)$ is trivial. For $(2)\implies(3)$, suppose that $f$ is norm-preserving. Then the norms of $x\in\quadalg{t'}{n'}$ and $f(x)=ux+c\in\quadalg[\A]{t}{n}$ are equal; namely
 \begin{align*}
  n' &= s_2(ux+c)\\
  &= s_2(ux) + s_{1,1}(ux,c) + s_2(c)\\
  &= u^2 s_2(x) + uc s_{1,1}(x,1) + c^2 s_2(1)\\
  &= u^2 n + uct + c^2.
 \end{align*}
 Now note that in any quadratic algebra, the trace of any element $d$ is $s_{1,1}(d,1) = s_2(d+1) - s_2(d) - 1$, so norm-preserving maps are also trace-preserving. Therefore the traces of $x\in\quadalg{t'}{n'}$ and $f(x)=ux+c\in\quadalg[\A]{t}{n}$ are also equal:
 \begin{align*}
  t' &= s_{1,1}(ux+c,1)\\
  &= u s_{1,1}(x, 1) + c s_{1,1}(1,1)\\
  &= ut + 2c.
 \end{align*}
Finally, for $(3)\implies(1)$ suppose that $t'=ut+2c$ and $n' = u^2n + uct + c^2$, so that $p_{ux+c}(\lambda) = \lambda^2 - t'\lambda+n'$ is the characteristic polynomial of $ux+c\in\quadalg{t}{n}$. But for any element $b$ of any rank-$n$ $\A$-algebra $\B$, the homomorphism $\A[x]/(p_b(x))\to\B$ sending $x$ to $b$ is universally norm-preserving by \cite{BieselGioia}.
\end{proof}

\begin{remark}
 In particular, if $\quadalg[\A]{t}{n}$ is isomorphic to $\quadalg[\A]{t'}{n'}$, there must exist $u,c\in\A$ with $u$ a unit, such that $t'=ut+2c$ and $n'=u^2t+utc+c^2$.
\end{remark}

\section{Norms of Based Quadratic Algebras}

We will now be considering rank-$n$ algebras $\A\to\B$ and the relationships between quadratic algebras over $\A$ and over $\B$. As a general convention for ease of context, we will tend to use lowercase letters like $s, t, m,u,c,\dots$ for elements of $\A$, and capital letters like $S, T, M, U, C,\dots$ for elements of $\B$.

As a bit of motivation for the definition of the norm of a quadratic algebra, suppose we have a rank-$n$ algebra $\A\to\B$ and a based quadratic $\B$-algebra $\quadalg[\B]{T}{N}$. We want the discriminant of the norm of $\quadalg{T}{N}$ to be the norm of the discriminant of $\quadalg[\B]{T}{N}$, namely $s_n(T^2 - 4N)$. So if we have $\norm_{\B/\A}\quadalg[\B]{T}{N} = \quadalg[\A]{t}{m}$, we had better have the identity
\[t^2 - 4m = s_n(T^2 - 4N).\]
Is there a canonical way to write $s_n(T^2 - 4N)$ as a square minus a multiple of four? Yes: we can use the polarized forms of $s_n$ to expand out $s_n(T^2 - 4N)$:
\begin{align*}
 s_n( - 4N + T^2) &= \sum_{k=0}^n s_{k, n-k}(-4N, T^2)\\
 &= \sum_{k=0}^n (-4)^k s_{k, n-k}(N, T^2)\\
 &= s_n(T^2) + \sum_{k=1}^n (-4)^k s_{k, n-k}(N, T^2),\\
 \intertext{since $s_{0,n}(N, T^2) = s_n(T^2)$ by \cref{lem:polarized-identity} (degeneracy),}
 &= \bigl[s_n(T)\bigr]^2 - 4\left[\sum_{k=1}^n (-4)^{k-1} s_{k, n-k}(N, T^2)\right],
\end{align*}
where we have used \cref{lem:multiplicative-sn} to rewrite $s_n(T^2)$ as $s_n(T)^2$. If we then let $t = s_n(T)$ and $m = \sum_{k=1}^n (-4)^{k-1} s_{k, n-k}(N, T^2)$, we have $t^2 - 4m = s_n(T^2 - 4N)$ as desired, and $\quadalg[\A]{t}{m}$ is a good candidate for the norm of $\quadalg[\B]{T}{N}$. That is indeed the definition we will use for the norm of a based quadratic algebras:

\begin{definition}\label{def:norm-based}
 Let $\B$ be a rank-$n$ $\A$-algebra, and let $\quadalg[\B]TN$ be a based quadratic $\B$-algebra. Then we define the \emph{norm} of $\quadalg[\B]{T}{N}$ to be the based quadratic $\A$-algebra
 \[\norm_{\B/\A}\quadalg[\B]{T}{N} \coloneqq \bigquadalg[\A]{s_n(T)}{\sum_{k=1}^n (-4)^{k-1}s_{k, n-k}(N, T^2)}.\]
 Noting that over the polynomial ring $\A[\lambda]$ we have 
 \[s_n(\lambda N + T^2) = \sum_{k=0}^n \lambda^k s_{k, n-k}(N, T^2),\]
 and therefore
 \[\sum_{k=1}^n \lambda^{k-1} s_{k, n-k}(N, T^2) = \frac{s_n(\lambda N + T^2) - s_n(T^2)}{\lambda},\]
 we will sometimes write the norm of $\quadalg[\B]{T}{N}$ as 
 \[\norm_{\B/\A}\quadalg[\B]{T}{N} = \bigquadalg[\A]{s_n(T)}{\left.\frac{s_n(\lambda N + T^2) - s_n(T^2)}{\lambda}\right|_{\lambda=-4}}\]
 to denote that the norm entry can be calculated as that fraction in $\A[\lambda]$ and then evaluated in $\A$ along the map $\A[\lambda]\to\A:\lambda\mapsto -4$.
\end{definition}

\begin{example}
 Let $\A$ be a ring in which $2$ is a unit, such as a number field. Then by completing the square, all based quadratic algebras over $\A$ can be written in the form $\A[\sqrt{d}] = \A[x]/(x^2-d) = \quadalg[\A]{0}{-d}$. What are the norms of based quadratic algebras of this type? If we have a rank-$n$ algebra $\A\to \B$ (such as a degree-$n$ extension of number fields), and a quadratic $\B$-algebra $\B[\sqrt{D}]$, then its norm is the quadratic $\A$-algebra
 \begin{align*}
  \norm_{\B/\A}\quadalg[\B]{0}{-D} &= \bigquadalg[\A]{s_n(0)}{\left.\frac{s_n(\lambda(-D) + (0)^2) - s_n((0)^2)}{\lambda}\right|_{\lambda=-4}}\\
  &= \bigquadalg[\A]{0}{\left.\frac{s_n(-\lambda D)}{\lambda}\right|_{\lambda=-4}}\\
  &= \quadalg[\A]{0}{\left.(-1)^n \lambda^{n-1} s_n(D)\right|_{\lambda=-4}}\\
  &= \quadalg[\A]{0}{-4^{n-1} s_n(D)}\\
  &= \A[\sqrt{4^{n-1} s_n(D)}] = \A[2^{n-1} \sqrt{s_n(D)}]\\
  &\cong \A[\sqrt{s_n(D)}]
 \end{align*}
since $2$ is a unit in $\A$. So the norm of $\B[\sqrt{D}]$ is $\A[\sqrt{s_n(D)}]$.
\end{example}

Next we have a few examples that begin to show how this norm operation is related to the monoid operation $\ast$ on quadratic algebras in \cite{Voight}, which is defined locally by

\begin{equation}\label{eq:ast}
\begin{gathered}
 \A[x]/(x^2 - sx + m) \ast \A[x]/(x^2 - tx + n) \coloneqq \\
 \A[x]/(x^2 - (st)x + (mt^2 + ns^2 - 4mn)),
 \end{gathered}
\end{equation}
i.e.\ $\quadalg{s}{m} \ast \quadalg{t}{n} \coloneqq \quadalg{st}{mt^2 + ns^2 - 4mn}$. This monoid operation is (up to isomorphism) commutative and associative, with a two-sided identity given by the split algebra $\quadalg[\A]{1}{0}\cong \A\times \A$ and a zero (absorbing) element given by the degenerate algebra $\quadalg[\A]{0}{0} \cong \A[\eps]/(\eps^2)$.

\begin{example}
 If $\A\to\B$ is an arbitrary rank-$n$ algebra, then we have $\norm_{\B/\A}\quadalg[\B]{1}{0} = \quadalg[\A]{1}{0}$ and $\norm_{\B/\A}\quadalg[\B]{0}{0} = \quadalg[\A]{0}{0}$. In other words, the norm of the split quadratic algebra $\B\times\B$ is again the split algebra $\A\times\A$, and the norm of the degenerate quadratic algebra $\B[\eps]/(\eps^2)$ is the degenerate algebra $\A[\eps]/(\eps^2)$. This is part of what it means for the norm functor on quadratic algebras to be a homomorphism of monoids-with-zero under $\ast$.
\end{example}

We will prove in \cref{thm:homomorphism} that the norm functor is truly a monoid homomorphism, and there we will use the following example showing how the monoid operation $\ast$ is actually a special case of the norm:

\begin{example}\label{ex:norm-product}
Given two based quadratic $\A$-algebras $\quadalg[\A]SM$ and $\quadalg[\A]TN$, we can consider their ordinary cartesian product as a based quadratic $\A\times\A$-algebra 
\[\quadalg[\A]SM \times \quadalg[\A]TN \cong \quadalg[\A\times\A]{(S,T)}{(M,N)}.\]
 Since $\A\to\A\times\A$ is an algebra of rank $n=2$, we can take the norm of this quadratic $\A\times\A$ algebra to get a quadratic $\A$-algebra:
\begin{align*}
 \norm\quadalg{(S,T)}{(M,N)} &= \quadalg{s_2((S,T))}{s_{1,1}((M,N), (S,T)^2) - 4 s_2((M,N))}
\end{align*}
To evaluate these entries we look at the matrices by which $(S,T)$ and $(M,N)$ act on $\A\times\A$ with respect to the standard basis $\{(1,0), (0,1)\}$. They act by diagonal matrices:
\[(S,T)\colon \begin{pmatrix}S & 0\\0 & T\end{pmatrix}\qquad (M,N)\colon \begin{pmatrix}M & 0 \\ 0 & N\end{pmatrix}\]
so we have $s_2((S,T)) = \det \begin{psmallmatrix}S & 0\\0 & T\end{psmallmatrix} = ST$ and $s_2((M,N)) = \det \begin{psmallmatrix}M & 0\\0 & N\end{psmallmatrix} = MN$.

We also need $s_{1,1}((M,N), (S,T)^2)$, which is the coefficient of $\lambda\mu$ in $s_2(\lambda(M,N) + \mu(S,T)^2)$, i.e. the coefficient of $\lambda\mu$ in
\begin{align*}
\det\begin{pmatrix} \lambda M + \mu S^2 & 0 \\ 0 & \lambda N + \mu T^2\end{pmatrix} &= (\lambda M + \mu S^2)(\lambda N + \mu T^2)\\
&= \lambda^2 MN + \lambda\mu(MT^2 + NS^2) + \mu^2 S^2T^2, 
\end{align*}
so $s_{1,1}((M,N), (S,T)^2) = MT^2 + NS^2$.
Therefore the norm of $\quadalg{S}{M}\times\quadalg{T}{N}$ is
\begin{align*}
 \norm\quadalg{(S,T)}{(M,N)} &= \quadalg{s_2((S,T))}{s_{1,1}((M,N), (S,T)^2) - 4 s_2((M,N))}\\
 &= \quadalg{ST}{MT^2 + NS^2 - 4MN},
\end{align*}
which is the composite quadratic algebra $\quadalg{S}{M} \ast \quadalg{T}{N}$ defined by the monoid operation $\ast$ in \cref{eq:ast}.
so the norm of a product of two quadratic algebras is their composite under the monoid operation $\ast$. Similarly, the norm of the quadratic $\A\times\A\times\A$-algebra $\quadalg{S}{M} \times \quadalg{T}{N}\times \quadalg{U}{P}$ would be the quadratic $\A$-algebra $\quadalg{S}{M} \ast \quadalg{T}{N}\ast \quadalg{U}{P}$, and so on.
\end{example}

Next we show that the norm operation is invariant under base change and is transitive along towers of algebras:

\begin{theorem}\label{thm:basechange-based}
 The norm operation on based quadratic algebras commutes with base change: Let $\A\to\B$ be a rank-$n$ algebra, and let $\quadalg[\B]TN$ be a based quadratic $\B$-algebra. If $\C$ is any $\A$-algebra, then $\C\to\C\otimes_\A \B$ is another rank-$n$ algebra, and
 \[\norm_{(\C\otimes_\A \B)/\C} \quadalg{T}{N} \cong \C\otimes_\A \norm_{\B/\A}\quadalg TN.\]
 \end{theorem}

\begin{proof}
 Since $s_n$ is a polynomial law, it and its polarized forms commute with base change, so the two based quadratic $\C$-algebras agree exactly.
\end{proof}

\begin{theorem}\label{thm:transitive-based}
 The norm operation on based quadratic algebras is transitive: If $\A\to \B$ is a rank-$n$ algebra and $\B\to \C$ is a rank-$m$ algebra, and if $\quadalg[\C]{T}{N}$ is a based quadratic $\C$-algebra, then $\norm_{\C/\A}(\quadalg[\C]{T}{N}) \cong \norm_{\B/\A}(\norm_{\C/\B}(\quadalg[\C]{T}{N}))$ as based quadratic $\A$-algebras.
\end{theorem}

\begin{proof}
 We check that the traces and norms of the canonical generators of $\norm_{\C/\A}\quadalg{T}{N}$ and $\norm_{\B/\A}\norm_{\C/\B}\quadalg{T}{N}$ agree. First, the traces of $x\in\norm_{\B/\A}\norm_{\C/\B}\quadalg{T}{N}$ and $x\in\norm_{\C/\A}\quadalg{T}{N}$ are $s_n(s_m(T))= s_{mn}(T)$. 
  Second, the norm of $x\in\norm_{\B/\A}\norm_{\C/\B}\quadalg{T}{N} %= \norm_{\B/\A}\bigquadalg{s_m(T)}{\left.\frac{s_m(\lambda N + T^2) - s_m(T^2)}{\lambda}\right|_{\lambda=-4}}
 $  is
 \begin{gather*}
  \left.\frac{s_n\left(\mu\left[\frac{s_m(\lambda N + T^2) - s_m(T^2)}{\lambda}\right]_{\lambda=-4}+s_m(T)^2\right) - s_n(s_m(T)^2)}{\mu}\right|_{\mu=-4},\\
  \intertext{in which, since we are setting $\mu$ equal to $-4$ anyway, we might as well set $\lambda$ equal to $\mu$ instead of $-4$:}
  \begin{aligned} 
  &=  \left.\frac{s_n\left(\mu\left[\frac{s_m(\lambda N + T^2) - s_m(T^2)}{\lambda}\right]_{\lambda=\mu}+s_m(T)^2\right) - s_n(s_m(T)^2)}{\mu}\right|_{\mu=-4}\\
  &= \left.\frac{s_n\left(\mu\left[\frac{s_m(\mu N + T^2) - s_m(T^2)}{\mu}\right]+s_m(T)^2\right) - s_n(s_m(T)^2)}{\mu}\right|_{\mu=-4}\\
  &= \left.\frac{s_n\left([s_m(\mu N + T^2) - s_m(T^2)]+s_m(T)^2\right) - s_n(s_m(T)^2)}{\mu}\right|_{\mu=-4}\\
&= \left.\frac{s_{mn}(\mu N + T^2) - s_{mn}(T^2)}{\mu}\right|_{\mu=-4},\\
\end{aligned}
\end{gather*}
which is the norm of $x$ in $\norm_{\C/\A}\quadalg{T}{N}$, as desired.
\end{proof}

\cref{thm:transitive-based} may be understood as a kind of functoriality of the norm map with respect to the underlying algebras. In order to extend the norm operation to all quadratic algebras, we will also need a functoriality with respect to isomorphisms of based quadratic algebras. In \cref{prop:different-generator} we define how the norm map acts on general norm-preserving homomorphisms of quadratic algebras, and in \cref{prop:functorial} we prove that this norm map preserves composition of homomorphisms. The combinatorial argument at the heart of \cref{prop:different-generator} is the key result that makes this norm functor well-defined.

\begin{proposition}\label{prop:different-generator}
 Let $\B$ be a rank-$n$ $\A$-algebra, and let $\quadalg TN$ and $\quadalg{T'}{N'}$ be based quadratic $\B$-algebras. If we have an isomorphism (resp.\ norm-preserving homomorphism) of $\B$-algebras
\[\begin{gathered}
f\colon \quadalg{T'}{N'} \to \quadalg TN\\
: x \mapsto Ux + C,
\end{gathered}\]
  then we also have an isomorphism (resp.\ norm-preserving homomorphism)
 \begin{equation}\label{eq:hom-image}
  \begin{gathered}
\norm_{\B/\A}(f)\colon \norm_{\B/\A}\quadalg{T'}{N'} \to \norm_{\B/\A}\quadalg TN\\
: x \mapsto s_n(U) x + \sum_{k=1}^n 2^{k-1} s_{k, n-k}(C, UT).
\end{gathered}
\end{equation}
We will also write the constant term above as $\left.\frac{s_n(\lambda C + UT) - s_n(UT)}{\lambda}\right|_{\lambda=2}$, to indicate that it is the image under the map $\A[\lambda]\to \A:\lambda\mapsto 2$ of the element $(s_n(\lambda C + UT) - s_n(UT))/\lambda$.
\end{proposition}

\begin{proof}
We prove the more general claim in the case that $f$ is a norm-preserving homomorphism, noting that if $f$ is an isomorphism, then $U$ is a unit in $\B$, so $s_n(U)$ is a unit in $\A$ and $\norm_{\B/\A}(f)$ is an isomorphism too.

First let $\tilde T = UT$ and $\tilde N = U^2N$. We can then decompose the norm-preserving map $\quadalg{T'}{N'}\to\quadalg{T}{N}$ into two maps $\quadalg{T'}{N'}\to\quadalg{\tilde T}{\tilde N}$ sending $x\mapsto x+C$ and $\quadalg{\tilde T}{\tilde N}\to\quadalg{T}{N}$ sending $x\mapsto Ux$. We will show immediately after this proposition that the norm operation on norm-preserving homomorphisms is functorial; it therefore suffices to show that \cref{eq:hom-image} provides a norm-preserving homomorphism between the norm algebras in the cases that $C=0$ or $U=1$.

 First consider the case $C=0$, so that we are considering the map $\quadalg{UT}{U^2 N}\to \quadalg{T}{N}$ sending $x\mapsto Ux$. We must show that the map $\norm\quadalg{UT}{U^2N}\to \norm\quadalg{T}{N}$ sending $x\mapsto s_n(U)x$ is also norm-preserving. Indeed, the trace of $x\in\norm\quadalg{UT}{U^2N}$ is $s_n(UT) = s_n(U)s_n(T)$, which is also the trace of $s_n(U)x\in \norm\quadalg{T}{N}$. And their norms match as well: the norm of $x\in \norm\quadalg{UT}{U^2N}$ is
 \begin{align*}
  \sum_{k=1}^n (-4)^{k-1} s_{k, n-k}(U^2N, (UT)^2) &= \sum_{k=1}^n (-4)^{k-1} s_{k, n-k}(U^2N, U^2T^2)\\
  &= \sum_{k=1}^n (-4)^{k-1} s_n(U^2) s_{k, n-k}(N,T^2)\\
  &= s_n(U)^2\sum_{k=1}^n (-4)^{k-1} s_{k, n-k}(N,T^2),
 \end{align*}
 which is the norm of $s_n(U)x\in\norm\quadalg{T}{N}$.

Now we consider the case $U=1$ but $C\neq 0$, so that we are considering the norm-preserving map $\quadalg{T+2C}{N+CT+C^2} \to \quadalg{T}{N}$ sending $x\mapsto x+C$, and we wish to show that the map $\norm\quadalg{T+2C}{N+CT+C^2}\to\norm\quadalg{T}{N}$ sending $x\mapsto x+\sum_{k=1}^n 2^{k-1} s_{k, n-k}(C, T)$ is also norm-preserving. The trace of $x$ in $\norm\quadalg{T+2C}{N+CT+C^2}$ is
\begin{align*}
 s_n(T+2C) &= \sum_{k=0}^n s_{k, n-k}(2C, T)\\
 &= \sum_{k=0}^n 2^k s_{k, n-k}(C, T)\text{ by homogeneity}\\
 &= s_n(T) + 2\sum_{k=1}^n 2^{k-1} s_{k, n-k}(C, T)
\end{align*}
which equals the trace of $x + \sum_{k=1}^n 2^{k-1} s_{k, n-k}(C, T)$ in $\norm\quadalg{T}{N}$ as desired.

Comparing norms will be more difficult; we will show separately that the norm of $x$ in $\norm\quadalg{T+2C}{N+CT+C^2}$ and the norm of $x+\sum_{k=1}^n 2^{k-1} s_{k, n-k}(C, T)$ in $\norm\quadalg{T}{N}$ are both equal to
\begin{equation}\label{eq:norm-translate}
 \sum_{i=1}^n (-4)^{i-1} s_{i, n-i}(N, T^2) + \sum_{\substack{j,k\in \NN:\\1 \leq j+k \leq n}} 4^{j+k-1} s_{j, k, n-j-k}(C^2, CT, T^2).
\end{equation}

 Consider first the norm of $x$ in $\norm\quadalg{T+2C}{N+CT+C^2}$. By definition it equals
 \begin{align*}
  &\sum_{m=1}^n (-4)^{m-1} s_{m, n-m}(N+CT+C^2, (T+2C)^2)\\
  &= \sum_{m=1}^n (-4)^{m-1} s_{m, n-m}(N+CT+C^2, T^2+4CT+4C^2)\\
  &= \sum_{m=1}^n(-4)^{m-1}\sum_{\substack{p,q,r,s,t,u\in\NN\\p+q+r=m\\s+t+u=n-m}}s_{p,q,r,s,t,u}(N, C^2, CT, 4C^2, 4CT, T^2)\\
  &= \sum_{m=1}^n(-4)^{m-1}\sum_{\substack{p,q,r,s,t,u\in\NN\\p+q+r=m\\s+t+u=n-m}}4^{s+t}s_{p,q,r,s,t,u}(N, C^2, CT, C^2, CT, T^2)\\
  &= \sum_{m=1}^n\sum_{\substack{p,q,r,s,t,u\in\NN\\p+q+r=m\\s+t+u=n-m}}(-4)^{m-1}4^{n-m-u}\binom{q+s}{s}\binom{r+t}{t}s_{p,q+s,r+t,u}(N, C^2, CT, T^2)
 \end{align*}
 by \cref{lem:polarized-identity} (combination).
We may then ask what the coefficient is of a given $s_{i,j,k,\ell}(N, C^2, CT, T^2)$ in the last sum, for natural numbers $i+j+k+\ell=n$. That coefficient is
\begin{align*}
 &\sum_{m=1}^n \sum_{\substack{q,r,s,t\in\NN\\i+q+r=m\\s+t+\ell=n-m\\q+s=j\\r+t=k}}(-4)^{m-1}4^{n-m-\ell}\binom{j}{s}\binom{k}{t}\\
 &= 4^{n-\ell-1}\sum_{m=1}^n (-1)^{m-1}\sum_{\substack{q,r,s,t\in\NN\\q+r=m-i\\s+t=n-m-\ell\\q+s=j\\r+t=k}}\binom{j}{s}\binom{k}{t}
\end{align*}
That inner sum over $q,r,s,t\in\NN$ is exactly counting the number of ways of forming a subset of size $m-i$ from a disjoint union of two sets of size $j$ and $k$: you must divide the set of size $j$ into sets of size $q$ and $s$ (for which there are $\binom{j}{s}$ possibilities), and the set of size $k$ into sets of size $r$ and $t$ (with $\binom{k}{t}$ possibilities), such that $q+r$ is the desired size $m-i$. Therefore that inner sum merely equals $\binom{j+k}{m-i}$, so the coefficient of $s_{i,j,k,\ell}(N,C^2,CT,T^2)$ is
\begin{align*}
 &= 4^{n-\ell-1}\sum_{m=1}^n (-1)^{m-1} \binom{j+k}{m-i}\\
 &= 4^{n-\ell-1}\sum_{m'=1-i}^{n-i} (-1)^{m'+i-1} \binom{j+k}{m'},
\end{align*}
where we have changed the index of summation to $m'=m-i$. Now that sum over $m'$ possibly includes several terms where the binomial coefficient vanishes because $m'$ is out of bounds. In particular, $n-i$ is always greater than or equal to $j+k$, so we might as well stop the sum at $m'=j+k$. Similarly, we may as well start the sum at $m'=0$, unless $i=0$, in which case the sum starts at $1$. So if $i>0$, the coefficient of $s_{i,j,k,\ell}(N, C^2, CT, T^2)$ is
\begin{align*}
&4^{n-\ell-1}(-1)^{i-1} \sum_{m'=0}^{j+k} (-1)^{m'}\binom{j+k}{m'}\\
&= 4^{n-\ell-1}(-1)^{i-1} (1-1)^{j+k}\\
 &= \begin{cases}0 &\text{if }j+k > 0\\
 4^{n-\ell-1}(-1)^{i-1} = (-4)^{i-1} &\text{if }j=k=0.
\end{cases}
\end{align*}
On the other hand, if $i=0$, the coefficient of $s_{0, j, k, \ell}(N,C^2, CT, T^2)$ is
\begin{align*}
 &\qquad 4^{n-\ell-1}(-1)\sum_{m'=1}^{j+k} \binom{j+k}{m'}\\
 &= -4^{n-\ell-1}[(1-1)^{j+k} - 1]\\
 &= \begin{cases}
  4^{n-\ell-1}=4^{j+k-1} &\text{if }j+k>0\\
  0 &\text{if }j=k=0.
 \end{cases}
\end{align*}
Putting it all together, and leaving out the terms with coefficient $0$, the norm of $x$ in $\norm\quadalg{T+2C}{N+CT+C^2}$ is
\[\sum_{i=1}^n (-4)^{i-1} s_{i, n-i}(N, T^2) + \sum_{\substack{j,k\in\NN\\1\leq j+k\leq n}} 4^{j+k-1}s_{j, k, n-j-k}(C^2, CT, T^2),\]
which matches \cref{eq:norm-translate} as desired.

Now we must show that this also agrees with the norm of $x+\sum_{k=1}^n 2^{k-1} s_{k, n-k}(C, T)$ in $\norm\quadalg{T}{N}$. This norm is
\begin{align*}
&\sum_{k=1}^n (-4)^{k-1} s_{k, n-k}(N, T^2)+ s_n(T)\left(\sum_{k=1}^n 2^{k-1} s_{k, n-k}(C,T)\right) + \left(\sum_{k=1}^n 2^{k-1} s_{k, n-k}(C,T)\right)^2.
\end{align*}
The first term agrees with that in \cref{eq:norm-translate}, so we must show that the second and third terms add up to $\sum_{j,k\in\NN: 1\leq j+k\leq n} 4^{j+k-1}s_{j,k,n-j-k}(C^2, CT, T^2)$. We begin by expanding the third term as
\begin{align*}
 \left(\sum_{k=1}^n 2^{k-1} s_{k, n-k}(C,T)\right)^2 &= \sum_{k, \ell=1}^n 2^{k+\ell - 2} s_{k, n-k}(C, T)s_{\ell, n-\ell}(C, T).
\end{align*}

Now $s_{k, n-k}(C, T)$ is the coefficient of $\lambda^k$ in $s_n(\lambda C + T)$, so $s_{k, n-k}(C, T)s_{\ell, n-\ell}(C, T)$ is the coefficient of $\lambda^k\mu^\ell$ in 
\begin{align*}
 s_n(\lambda C + T)s_n(\mu C + T) &= s_n(\lambda\mu C^2 + (\lambda+\mu)CT + T^2)\\
 &= \sum_{\substack{a,b,c\in\NN\\a + b + c = n}} s_{a,b,c}(\lambda\mu C^2, (\lambda+\mu)CT, T^2)\\
 &= \sum_{\substack{a,b,c\in\NN\\a + b + c = n}} (\lambda\mu)^a(\lambda+\mu)^b s_{a,b,c}(C^2, CT, T^2)\\
 &= \sum_{\substack{a,b,c\in\NN\\a + b + c = n}} \sum_{d=0}^b \lambda^{a+d}\mu^{a+b-d}\binom{b}{d}s_{a,b,c}(C^2, CT, T^2)
\end{align*}
whose coefficient of $\lambda^k\mu^\ell$ is therefore
\begin{align*}
s_{k,n-k}(C,T)s_{\ell,n-\ell}(C,T) &= \sum_{\substack{a,b,c,d\in\NN\\a+b+c=n\\a+d=k\\a+b-d=\ell}} \binom{b}{d}s_{a,b,c}(C^2, CT, T^2)\\
 &= \sum_{\substack{a,b,c\in\NN\\a+b+c=n\\k+\ell=2a+b}} \binom{b}{k-a}s_{a,b,c}(C^2, CT, T^2),
\end{align*}
so we have
\begin{align*}
 \left(\sum_{k=1}^n 2^{k-1} s_{k, n-k}(C,T)\right)^2 &= \sum_{k, \ell=1}^n 2^{k+\ell - 2} s_{k, n-k}(C, T)s_{\ell, n-\ell}(C, T).\\
 &= \sum_{k, \ell=1}^n 2^{k+\ell - 2} \sum_{\substack{a,b,c\in\NN\\a+b+c=n\\k+\ell=2a+b}} \binom{b}{k-a}s_{a,b,c}(C^2, CT, T^2)\\
 &= \sum_{\substack{a,b,c\in\NN\\a+b+c=n}} 2^{2a+b - 2} s_{a,b,c}(C^2, CT, T^2)\sum_{\substack{1\leq k, \ell \leq n\\k+\ell=2a+b}} \binom{b}{k-a}\\
 &= \sum_{\substack{a,b,c\in\NN\\a+b+c=n}} 2^{2a+b - 2} s_{a,b,c}(C^2, CT, T^2)\sum_{k=1}^{ 2a+b-1} \binom{b}{k-a}\\
\end{align*}
Now the value of that inner sum breaks into three cases depending on how many terms are missing from the sum $2^b = \sum_j \binom{b}{j}$: either $a = 0$ and $b > 0$, in which case we are missing the start and the end terms, giving
\[\sum_{k=1}^{2a+b-1}\binom{b}{k-a} = \sum_{k=1}^{b-1}\binom{b}{k} = 2^b - 2,\]
or $a=0$ and $b=0$, in which case the sum is empty and equals $0$, or $a>0$, in which case neither end term is missing:
\[\sum_{k=1}^{2a+b-1}\binom{b}{k-a} = \sum_{k'=1-a}^{a+b-1}\binom{b}{k'}=\sum_{k'=0}^{b}\binom{b}{k'}= 2^b.\]
So our expansion of $(\sum_{k=1}^n 2^{k-1}s_{k,n-k}(C,T))^2$ becomes
\begin{align*}
&\qquad \sum_{\substack{a,b,c\in\NN\\a+b+c=n}} 2^{2a+b - 2} s_{a,b,c}(C^2, CT, T^2)\sum_{k=1}^{ 2a+b-1} \binom{b}{k-a}\\
&= \sum_{\substack{a,b,c\in\NN\\a+b+c=n\\a+b\geq 1}} 2^{2a+b-2}(2^b) s_{a,b,c}(C^2, CT, T^2) - \sum_{\substack{a,b,c\in\NN\\a+b+c=n\\a=0,b>0}} 2^{b-2} (2) s_{a,b,c}(C^2,CT, T^2)\\
&= \sum_{\substack{a,b,c\in\NN\\a+b+c=n\\a+b\geq 1}} 2^{2a+2b-2}s_{a,b,c}(C^2, CT, T^2) - \sum_{b=1}^n 2^{b-1}s_{b,n-b}(CT, T^2).
\end{align*}
Therefore the sum $(\sum_{k=1}^n 2^{k-1}s_{k,n-k}(C,T))^2 + \sum_{b=1}^n 2^{b-1}s_{b,n-b}(CT, T^2)$ equals
\[\sum_{\substack{a,b,c\in\NN\\a+b+c=n\\a+b\geq 1}} 2^{2a+2b-2}s_{a,b,c}(C^2, CT, T^2) = \sum_{\substack{a,b\in\NN\\1\leq a+b\leq n}} 4^{a+b-1}s_{a,b,c}(C^2, CT, T^2)\]
as desired. So the map $\norm\quadalg{T+2C}{N+CT+C^2}\to\norm\quadalg{T}{N}$ sending $x\mapsto x + \sum_{k=1}^n 2^{k-1}s_{k,n-k}(C,T)$ preserves traces and norms, so it is a norm-preserving homomorphism.
\end{proof}

\begin{example}\label{ex:swap}
 For example, let $\A\to\B$ be a rank-$n$ algebra and let $f:\B^2\to\B^2$ be the ``swap'' map of quadratic algebras $\B$-algebras sending $(b, b')\mapsto (b',b)$.  (In terms of based quadratic algebras, this is the map $\quadalg[\B]{1}{0}\to\quadalg[\B]{1}{0}$ sending $x\mapsto -x+1$.)  If $n$ is odd, then $\norm_{\B/\A}(f):\A^2\to\A^2$ is also the swap map; if $n$ is even it is the identity map.
\end{example}

\begin{proof}
 Applying the norm map to $f\colon \quadalg[\B]{1}{0}\to\quadalg[\B]{1}{0}: x\mapsto -x+1$, we obtain
 \begin{align*}
  \norm_{\B/\A}(f)\colon x&\mapsto s_n(-1)x + \left.\frac{s_n(\lambda - 1) - s_n(-1)}{\lambda}\right|_{\lambda=2}\\
  &= (-1)^n x + \left.\frac{(\lambda-1)^n - (-1)^n}{(\lambda-1) - (-1)}\right|_{\lambda=2}\\
  &= (-1)^n x + \left.(\lambda-1)^{n-1} + (-1)(\lambda-1)^{n-2} + \dots + (-1)^{n-1}\right|_{\lambda=2}\\
  &= (-1)^n x + 1 - 1 + \ldots + (-1)^{n-1}.
 \end{align*}
Now if $n$ is even, this works out to $x\mapsto 1x+0$, the identity map. And if $n$ is odd, this becomes $-x + 1$, the swap map $\A^2\to\A^2$.
\end{proof}

\begin{proposition}\label{prop:functorial}
 The assignment $f\mapsto \norm_{\B/\A}(f)$ in \cref{prop:different-generator} is functorial: if we have a commuting triangle of univerally norm-preserving homomorphisms between quadratic $\B$-algebras
   \[\begin{array}{ccccc}
  & &  \quadalg{T'}{N'} & &\\
  & \nearrow & & \searrow & \\
  \quadalg{T''}{N''} & & \longrightarrow & & \quadalg{T}{N}
 \end{array}\]
 then the resulting triangle of norm-preserving $\A$-algebra homomorphisms
  \[\begin{array}{ccccc}
  & & \norm_{\B/\A} \quadalg{T'}{N'} & &\\
  & \nearrow & & \searrow & \\
  \norm_{\B/\A}\quadalg{T''}{N''} & & \longrightarrow & & \norm_{\B/\A}\quadalg{T}{N}
 \end{array}\]
 also commutes.
\end{proposition}

\begin{proof}
Let us suppose that the homomorphism $\quadalg{T'}{N'} \to \quadalg{T}{N}$ sends $x\mapsto Ux+C$, so that $T'=UT+2C$. Suppose also that the homomorphism $\quadalg{T''}{N''}\to \quadalg{T'}{N'}$ sends $x\mapsto Vx+D$, so that the composite homomorphism $\quadalg{T''}{N''}\to\quadalg{T}{N}$ sends $x\mapsto V(Ux+C)+D = VUx + (VC+D)$. Then the composite homomorphism of norms
\[\norm\quadalg{T''}{N''} \to \norm\quadalg{T'}{N'} \to \norm\quadalg{T}{N}\]
sends
\begin{align*}
 x&\mapsto s_n(V)x +\left.\frac{s_n(\lambda D + VT') - s_n(VT')}{\lambda}\right|_{\lambda = 2}\\
 &= s_n(V)s_n(U)x + s_n(V)\left.\frac{s_n(\lambda C + UT) - s_n(UT)}{\lambda}\right|_{\lambda = 2}+\left.\frac{s_n(\lambda D + VT') - s_n(VT')}{\lambda}\right|_{\lambda = 2}\\
 &= s_n(VU)x + \left.\frac{s_n(\lambda VC + VUT) - s_n(VUT)}{\lambda}\right|_{\lambda = 2}+\left.\frac{s_n(\lambda D + VUT+2VC) - s_n(VUT + 2 VC)}{\lambda}\right|_{\lambda = 2}\\
\end{align*}
We can combine the two constant terms into one larger sum (replacing a couple of extra 2's by $\lambda$ since they will be set equal to $2$ again anyway) as follows:
\begin{align*}
 &\left.\frac{s_n(\lambda VC + VUT) - s_n(VUT)}{\lambda}\right|_{\lambda = 2}+\left.\frac{s_n(\lambda D + VUT+2VC) - s_n(VUT + 2 VC)}{\lambda}\right|_{\lambda = 2}\\
 &= \left.\frac{s_n(\lambda VC + VUT) - s_n(VUT) + s_n(\lambda D + VUT + \lambda VC) - s_n(VUT + \lambda VC)}{\lambda}\right|_{\lambda=2}\\
 &= \left.\frac{ s_n(\lambda D + VUT + \lambda VC) - s_n(VUT)}{\lambda}\right|_{\lambda=2}\\
\end{align*}
So in all, the composite map of norms $\norm\quadalg{T''}{N''}\to \norm\quadalg{T'}{N'}\to \norm\quadalg{T}{N}$ sends
\[x\mapsto s_n(VU) x + \left.\frac{ s_n(\lambda (VC+D) + VUT) - s_n(VUT)}{\lambda}\right|_{\lambda=2}\]
which is exactly the norm of the composite map $\quadalg{T''}{N''}\to\quadalg{T'}{N'}\to \quadalg{T}{N}$ sending $x\mapsto VUx + (VC+D)$, as desired.
\end{proof}

To show that the norm operation on general quadratic algebras is transitive, we will also need transitivity for the norm operation on homomorphisms:

\begin{proposition} \label{thm:transitive-morphism}The norm map on homomorphisms is also transitive: if $\A\to\B$ is a rank-$n$ algebra, and $\B\to\C$ is a rank-$m$ algebra, and $f:\quadalg{T'}{N'}\to\quadalg{T}{N}$ is a norm-preserving homomorphism of based quadratic $\C$-algebras, then
\[\norm_{\B/\A}(\norm_{\C/\B}(f)) = \norm_{\C/\A}(f).\]
\end{proposition}

\begin{proof}
 Let us suppose that $f:\quadalg{T'}{N'}\to\quadalg{T}{N}$ sends $x\mapsto Ux+C$. Then on the one hand, we have
 \[ \norm_{\C/\B}(f) \colon x\mapsto s_m(U)x +\left. \frac{s_m(\lambda C + UT) -s_m(UT)}{\lambda}\right|_{\lambda=2}\]
  so, since $x\in\norm_{\C/\B}\quadalg{T}{N}$ has trace $s_m(T)$, we have that $\norm_{\B/\A}( \norm_{\C/\B}(f))$ sends
  \begin{align*}
   x &\mapsto s_n(s_m(U))x +\left.\frac{s_n\left(\mu\left. \frac{s_m(\lambda C + UT) -s_m(UT)}{\lambda}\right|_{\lambda=2} + s_m(U)s_m(T)\right) - s_n(s_m(U)s_m(T))}{\mu}\right|_{\mu=2}\\
 &= s_{mn}(U) x +\left.\frac{s_n\left(\mu\left. \frac{s_m(\lambda C + UT) -s_m(UT)}{\lambda}\right|_{\lambda=\mu} + s_m(UT)\right) - s_n(s_m(UT))}{\mu}\right|_{\mu=2}\\
&= s_{mn}(U) x  +\left.\frac{s_n\left(\mu\frac{s_m(\mu C + UT) -s_m(UT)}{\mu} + s_m(UT)\right) - s_{mn}(UT)}{\mu}\right|_{\mu=2}\\
 &= s_{mn}(U)x + \left.\frac{s_n\bigl(s_m(\mu C+ UT)- s_m(UT)+s_m(UT)\bigr) - s_{mn}(UT)}{\mu}\right|_{\mu=2}\\
 &= s_{mn}(U)x + \left.\frac{s_{mn}(\mu C+ UT) - s_{mn}(UT)}{\mu}\right|_{\mu=2},
 \end{align*}
which is exactly where $\norm_{\C/\A}(f)$ sends $x$.
\end{proof}

\section{Norms of Quadratic Algebras}

We now have all the results we need in order to define the norm of a general quadratic algebra and prove that it is well-defined:

\begin{definition}\label{def:norm}
 Given a rank-$n$ $\A$-algebra $\B$, we define the norm of a general quadratic $\B$-algebra $\D$ as follows:
 \begin{enumerate}
  \item Choose elements $a_1,\dots, a_k\in \A$, together generating the unit ideal, such that each $\D_{a_i}$ is free as a $\B_{a_i}$-module.
  \item Choose generators for each $\D_{a_i}$ to obtain isomorphisms $\D_{a_i} \cong \quadalg[\B_{a_i}]{T_i}{N_i}$.
  \item Over overlaps $\A_{a_ia_j}$, use \cref{prop:different-generator} to convert the isomorphisms $\quadalg[\B_{a_ia_j}]{T_i}{N_i} \cong \D_{a_ia_j}\cong \quadalg[\B_{a_ia_j}]{T_j}{N_j}$ into isomorphisms \[\norm_{\B_{a_ia_j}/\A_{a_ia_j}}\quadalg[\B_{a_ia_j}]{T_i}{N_i} \cong \norm_{\B_{a_ia_j}/\A_{a_ia_j}}\quadalg[\B_{a_ia_j}]{T_j}{N_j}.\]
  \item Use these isomorphisms to glue the norms of the based quadratic $\A_{a_i}$-algebras $\norm_{\B_{a_i}/\A_{a_i}} \quadalg{T_i}{N_i}$ into a single quadratic $\A$-algebra, called $\norm_{\B/\A}(\D)$.
 \end{enumerate}
\end{definition}

\begin{lemma}
 The construction of $\norm_{\B/\A}(\D)$ in \cref{def:norm} is well-defined.
\end{lemma}

\begin{proof}
 There are two things to check: that the norms of the based quadratic algebras $\norm_{\B_{a_i}/\A_{a_i}} \quadalg{T_i}{N_i}$ do, in fact, glue together to yield a quadratic $\A$-algebra, and that the resulting algebra is independent (up to isomorphism) of the choices of the $a_i$ and generators for $\D_{a_i}$.
 
 First we check that the quadratic algebras do glue; this amounts to checking that over triple overlaps $\A_{a_ia_ja_k}$, the following triangle of isomorphisms commutes:
 \[\begin{array}{ccccc}
  & & \norm \quadalg{T_j}{N_j} & &\\
  & \nearrow & & \searrow & \\
  \norm\quadalg{T_i}{N_i} & & \longrightarrow & & \norm\quadalg{T_k}{N_k}
 \end{array}\]
 But this is true by \cref{prop:functorial} because the corresponding triangle of quadratic $\B_{a_ia_ja_k}$-algebra homomorphisms commutes, since the isomorphisms between them come from chosen isomorphisms with $\D_{a_ia_ja_k}$:
 \[\begin{array}{ccccc}
  & & \quadalg{T_j}{N_j} & &\\
  & \nearrow & & \searrow & \\
  \quadalg{T_i}{N_i} & & \longrightarrow & & \quadalg{T_k}{N_k}
 \end{array}\]
 
 Second, we check independence of the choices of the $a_i$ and generators for $\D_{a_i}$: suppose that we had a second collection of elements $a_1', a_2', \ldots$ and isomorphisms $\D_{a_i'}\cong \quadalg{T_i'}{N_i'}$. Then we could combine the two collections into one and construct a common gluing of all the $\norm\quadalg{T_i}{N_i}$ and $\norm\quadalg{T_i'}{N_i'}$. In particular, the result will be isomorphic to the gluings obtained from each family separately, so the norm is independent of choice of generator. 
\end{proof}

We can now boost our results on the norm operation for based quadratic algebras to this new, more general setting:

\begin{theorem}\label{thm:base-change}
 The norm operation on general quadratic algebras commutes with base change: Let $\A\to\B$ be a rank-$n$ algebra, and let $\D$ be a quadratic $\B$-algebra. If $\C$ is any $\A$-algebra, then $\C\to\C\otimes_\A \B$ is another rank-$n$ algebra, and $\C\otimes_\A \D$ is a quadratic $\C\otimes_\A \B$-algebra with
 \[\norm_{(\C\otimes_\A \B)/\C} (\C\otimes_\A \D) = \C \otimes_\A \norm_{\B/\A}(\D).\]
\end{theorem}

\begin{proof}
 We already know by \cref{thm:basechange-based} that this holds locally, so we need only check that taking the norms of the gluing isomorphisms also commutes with base change. But the gluing isomorphisms are also built out of polarized forms of the polynomial law $s_n$, so this is again automatic.
\end{proof}

\begin{theorem}\label{thm:transitive}
 The norm operation is transitive: If $\A\to \B$ is a rank-$n$ algebra and $\B\to \C$ is a rank-$m$ algebra, and if $\D$ is a quadratic $\C$-algebra, then $\norm_{\C/\A}(\D) \cong \norm_{\B/\A}(\norm_{\C/\B}(\D))$.
\end{theorem}

\begin{proof}
 We have already shown in \cref{thm:transitive-based} that this holds for norms of based quadratic algebras; it remains to show that the isomorphisms used to glue together the norms of based localizations also agree, but this holds by \cref{thm:transitive-morphism}.
\end{proof}

We can now use transitivity of the norm operation to show that $\norm_{\B/\A}$ is a homomorphism with respect to the monoid operation $\ast$ on quadratic algebras:

\begin{theorem}\label{thm:homomorphism}
 The norm operation is a homomorphism with respect to the monoid operation $\ast$ on isomorphism classes of quadratic algebras: If $\A\to \B$ is a rank-$n$ algebra and $\D$ and $\D'$ are quadratic $\B$-algebras, then $\norm_{\B/\A}(\D\ast \D') \cong \norm_{\B/\A}(\D) \ast \norm_{\B/\A} (\D')$.
\end{theorem}

\begin{proof}
 Consider the tower of algebras $\A \to \B \to \B^2$ and the quadratic $\B^2$-algebra $\D\times \D'$. We can take its norm to produce the quadratic $\A$-algebra
 \begin{align*}
  \norm_{\B^2/\A}(\D\times \D') &\cong \norm_{\B/\A}(\norm_{\B^2/\B}( \D\times \D' ))\\
  &\cong \norm_{\B/\A}( \D\ast \D').
 \end{align*}
 
 On the other hand, we can also factor the algebra map $\A \to \B^2$ as $\A \to \A^2 \to \B^2$, for which transitivity of the norm operation gives
 \begin{align*}
  \norm_{\B^2/\A}(\D\times \D') &\cong \norm_{\A^2/\A} (\norm_{\B^2/\A^2} (\D\times \D'))\\
  &\cong \norm_{\A^2/\A}(\norm_{\B/\A}(\D) \times \norm_{\B/\A}(\D'))\\
  &\cong \norm_{\B/\A}(\D) \ast \norm_{\B/\A}(\D').
 \end{align*}
 (In the second line, we have used base-change invariance of the norm operation to construct the norm over each factor separately, and in the third line, we have invoked \cref{ex:norm-product}) Comparing these two expressions for $\norm_{\B^2/\A}(\D\times \D')$, we obtain $\norm_{\B/\A}(\D\ast \D') \cong \norm_{\B/\A}(\D) \ast \norm_{\B/\A}(\D')$, as desired.
\end{proof}

Our last results are that this notion of norm of quadratic algebras commutes with taking determinant line bundles and discriminant bilinear forms:

\begin{theorem}\label{thm:norm-det}
 The norm operation commutes with taking determinant line bundles: If $\A\to\B$ is a rank-$n$ algebra and $\D$ is a quadratic $\B$-algebra, then $\bigwedge_\A^2\norm_{\B/\A}(\D) = \norm_{\B/\A}(\bigwedge_\B^2 \D)$.
\end{theorem}

\begin{proof}
 If $\L$ is a line bundle (rank-$1$ module) over $\B$, then its norm $\norm_{\B/\A}(\L)$ is the line bundle obtained by the following process analogous to \cref{def:norm}:
 \begin{enumerate}
  \item Choose a set of elements $a_i\in \A$ generating the unit ideal such that each localization $\L_{a_i}$ is free as an $\B_{a_i}$-module.
  \item Choose generators for each $\L_{a_i}$ to obtain isomorphisms $\phi_i:\L_{a_i}\cong \B_{a_i}$.
  \item Over overlaps $\A_{a_ia_j}$, we have isomorphisms $(\phi_i)_{a_j} \circ (\phi_j^{-1})_{a_i} \colon \B_{a_ia_j} \cong \L_{a_ia_j} \cong \B_{a_ia_j}$. Each such isomorphism is given by multiplication by a unit $U_{ij}\in\B_{a_ia_j}$; take the norm of $U_{ij}$ to get units $s_n(U_{ij})\in\A_{ij}$.
  \item Glue together free rank-1 $\A_{a_i}$-modules, using multiplication by $s_n(U_{ij})$ as the isomorphisms on overlaps, to get a single rank-$1$ $\A$-module, called $\norm_{\B/\A}(\L)$.
 \end{enumerate}

Now suppose $\L = \bigwedge^2 \D$. We will follow along steps 1 through 4 of \cref{def:norm} and see that $\bigwedge^2 \norm_{\B/\A}(\D)$ agrees with the above construction of $\norm_{\B/\A}(\L)$.

Since the determinant line bundle of a based quadratic algebra $\quadalg[\B]{T}{N}$ is free of rank $1$ with canonical generator $1\wedge x$, completing steps 1 and 2 of \cref{def:norm} also completes steps 1 and 2 of constructing the line bundle norm $\norm_{\B/\A}(\L)$. 
For step 3, suppose that for each $i,j\in\{1,\dots,k\}$ we have an isomorphism $\quadalg[\B_{a_ia_j}]{T_i}{N_i} \cong \D_{a_ia_k} \cong \quadalg[\B_{a_i}{a_j}]{T_j}{N_j}$ sending $x\mapsto U_{ij}x+C_{ij}$. 
Then taking exterior powers, we have isomorphisms $\bigwedge^2 \quadalg[\B_{a_ia_j}]{T_i}{N_i} \cong \bigwedge^2 \D_{a_ia_k} \cong \bigwedge^2 \quadalg[\B_{a_i}{a_j}]{T_j}{N_j}$ sending $1\wedge x \mapsto U_{ij} \cdot (1 \wedge x)$.
This means that multiplication by $U_{ij}$ is also the composite isomorphism $\B_{a_ia_j}\cong \L_{a_ia_j}\cong\B_{a_ia_j}$.

To construct $\norm_{\B/\A}(\D)$, we apply \cref{prop:different-generator} to the isomorphisms $\quadalg[\B_{a_ia_j}]{T_i}{N_i} \cong \quadalg[\B_{a_i}{a_j}]{T_j}{N_j}$ sending $x\mapsto U_{ij}x+C_{ij}$ to get isomorphisms
\[\norm\quadalg[\B_{a_ia_j}]{T_i}{N_i} \cong \norm\quadalg[\B_{a_i}{a_j}]{T_j}{N_j}: x\mapsto s_n(U_{ij})x+\left.\frac{s_n(\lambda C_{ij}+U_{ij}T_j)- s_n(U_{ij}T_j)}{\lambda}\right|_{\lambda=2}\]
which we use as gluing maps to construct $\norm_{\B/\A}(\D)$. Taking the exterior powers of these maps, we obtain isomorphisms sending $1\wedge x \mapsto s_n(U_{ij})\cdot (1\wedge x)$, meaning that $\bigwedge^2 \norm_{\B/\A}(\D)$ is obtained by gluing together free rank-$1$ $\A_{a_i}$-modules with the gluing isomorphisms given by multiplication by $s_n(U_{ij})$. That is exactly the gluing which produces $\norm_{\B/\A}(\L)$, as claimed.
\end{proof}

\begin{theorem}\label{thm:norm-disc}
 Taking the norm commutes with taking discriminants: if $\A\to\B$ is a rank-$n$ algebra and $\D$ is a quadratic $\B$-algebra with discriminant bilinear form $\disc_\D \colon (\extpower_\B^2 \D)^{\otimes 2} \to \B$, then the discriminant bilinear form of the quadratic $\A$-algebra $\norm_{\B/\A}( \D)$ agrees with the norm of $\disc_\D$ as a map of line bundles.
\end{theorem}

\begin{proof}
 Note that this only makes sense because by \cref{thm:norm-det} we have an isomorphism $\bigwedge^2\norm(\D) = \norm(\bigwedge^2 \D)$, so the norm of the line bundle homomorphism $\disc_\D \colon (\extpower^2 \D)^{\otimes 2} \to \B$ can be viewed as a homomomorphism $\norm(\disc_\D) \colon (\extpower^2 \norm(\D))^{\otimes 2} \to \A$. Then we can check locally that this agrees with $\disc_{\norm(\D)}$, so assume without loss of generality that $\D = \quadalg{T}{N}$.
 
 With respect to the canonical basis $1\wedge x$ for $\extpower^2\D$, we have that $\disc_\D$ is just multiplication by $T^2 - 4N$. The norm of this map is therefore multiplication by $s_n(T^2 - 4N)$.
 
 On the other hand, with respect to the canonical basis $\disc_{\norm(\D)}$ is multiplication by
 \begin{align*}
  & s_n(T)^2 - 4\left[\frac{s_n(\lambda N + T^2) - s_n(T^2)}{\lambda}\right]_{\lambda = -4} \\
  & = s_n(T)^2 + \left[s_n(\lambda N + T^2) - s_n(T^2)\right]_{\lambda = -4}\\
  &= s_n(T^2) + s_n(T^2 - 4N) - s_n(T^2)\\
  &= s_n(T^2 - 4N),
 \end{align*}
 which agrees with $\norm(\disc_\D)$ as desired.
\end{proof}

In particular, if a quadratic $\B$-algebra $\D$ is \'etale (its discriminant bilinear form is an isomorphism $(\extpower^2 \D)^{\otimes 2}\simto \B$), then its norm $\norm_{\B/\A}(\D)$ is also an \'etale quadratic $\A$-algebra. Since \'etale quadratic algebras are equivalent to $S_2$-torsors, we may ask if the norm operation agrees with the trace map for $S_2$-torsors:

\begin{theorem}\label{thm:etale-addition}
 On \'etale quadratic algebras, the norm operation reduces to the ordinary trace for $S_2$-torsors.
\end{theorem}

\begin{proof}
 Let $\A\to\B$ be a rank-$n$ algebra. Since a split \'etale $\B$-algebra is isomorphic to $\B^2\cong \quadalg[\B]10$, and $\norm_{\B/\A}\quadalg[\B]10 = \quadalg[\A]10\cong\A^2$, we know that since the norm operation commutes with base change, therefore the norm of any \'etale quadratic algebra is still \'etale. It remains to show that the norm of a $\B$-algebra automorphism of $\B^2$ is its trace as an automorphism of $\A^2$. Any $\B$-algebra automorphism of $\B^2$ is locally either the identity map $\B^2\to\B^2$ (which is sent to the identity map $\A^2\to\A^2$) or the swap map $\B^2\to\B^2$ sending $(b,b')\mapsto (b', b)$. By \cref{ex:swap}, the norm of the swap map is either the identity (if the rank $n$ is even) or again the swap map (if $n$ is odd). In other words, the norm of the swap map is the swap map composed with itself $n$ times, which is exactly the trace of the swap map, as desired.
\end{proof}

Finally, we have the following conjectured relationship between norms and discriminant algebras which was the motivation for constructing the norm functor for quadratic algebras in the first place:

\begin{conjecture}\label{conj:norm-discalg}
 Let $\Delta$ be the discriminant algebra operation of \cite{BieselGioia}, sending rank-$n$ algebras $\A\to\B$ to quadratic algebras $\A \to\Delta_{\B/\A}$. Now let $\A\to\B$ be a rank-$n$ algebra and $\B\to\C$ be a rank-$m$ algebra, so that $\A\to\C$ is also a rank-$mn$ algebra. Is it the case that
 \[\Delta_{\C/\A} \cong \norm_{\B/\A}(\Delta_{\C/\B}) \ast \Delta_{\B/\A}^{\ast m}?\]
\end{conjecture}

It is known that \cref{conj:norm-discalg} holds in case $\A\to\B$ and $\B\to\C$ are \'etale; see \cite[Theorem 4]{Waterhouse} for a proof.

\bibliographystyle{acm}
\bibliography{RefList}

\end{document}